\documentclass[a4paper, 11pt]{amsart}
\usepackage[utf8]{inputenc}
\usepackage[T1]{fontenc}
\usepackage{lmodern}
\usepackage[english]{babel}
\usepackage[dvipsnames]{xcolor}
\usepackage{adjustbox}
\usepackage{letltxmacro}
\usepackage{todonotes}
\LetLtxMacro\todonotestodo\todo
\renewcommand{\todo}[2][]{\todonotestodo[#1]{TODO: {#2}}}
\usepackage{enumerate}
\usepackage{hyperref}
\usepackage{breakurl}
\usepackage{url}
\usepackage{mathrsfs}
\usepackage{amssymb}
\usepackage{ stmaryrd }

\theoremstyle{definition}
\newtheorem{theorem}{Theorem}

\newtheorem{lemma}{Lemma}[section]
\newtheorem{proposition}[lemma]{Proposition}
\newtheorem{corollary}[lemma]{Corollary}

\newtheorem{remark}[lemma]{Remark}

\newtheorem{example}[lemma]{Example}

\newtheorem{definition}[lemma]{Definition}
\newtheorem*{claim*}{Claim}

\newtheorem*{theorem*}{Theorem}
\newtheorem*{corollary*}{Corollary}
\newtheorem*{lemma*}{Lemma}
\newtheorem*{remark*}{Remark}
\newtheorem*{question*}{Question}

\newcommand{\Z}{\mathbb{Z}}
\newcommand{\N}{\mathbb{N}}

\newcommand{\bdot}{\boldsymbol{\cdot}}

\title[On the arithmetic of ultraproducts]{On the arithmetic of ultraproducts of commutative cancellative monoids}


\author{Daniel Windisch}
\thanks{The author is supported by the Austrian Science Fund (FWF), project I~4406.}

\begin{document}

\maketitle

\begin{abstract}
We develop first steps in the study of factorizations of elements in ultraproducts of commutative cancellative monoids into irreducible elements. A complete characterization of the (multi-)sets of lengths in such objects is given. As applications, we show that several important properties from factorization theory cannot be expressed as first-order statements in the language of monoids, and we construct integral domains that realize every multiset of integers larger $1$ as a multiset of lengths. Finally, we give a new proof (based on our ultraproduct techniques) of a theorem by Geroldinger, Schmid and Zhong from additive combinatorics and we propose a general method for applying ultraproducts in the setting of non-unique factorizations.
\end{abstract}

\section{Introduction}

Having its origins primarily in algebraic number theory, see for instance~\cite{Geroldinger-Crelle} and the very recent work of Rago~\cite{Balint} into this direction, the theory of non-unique factorizations has become a standard topic in ring theory. Since the publication of the fundamental monograph by Geroldinger and Halter-Koch~\cite{GHK}, new and fruitful approaches in this area have been launched.

Several authors studied the multiplicative arithmetic of objects more geometric in flavour. Kainrath~\cite{KainrathDistribution} investigated prime divisors in the class groups of integral algebras over a field. This has been carried on by Fadinger-Held and the present author for affine toric varieties and their non-Noetherian analogues~\cite{KrullPrimeDiv,affine} and was then applied to the arithmetic of their affine coordinate rings in a joint paper with G.W. Chang~\cite[Section 5]{weaklyKrull2}. Lately, a first step to an understanding of non-unique factorizations in local complete intersection rings was taken~\cite{classgroup-complete-int}.

Smertnig opened the non-commutative world for the study of non-unique factorizations, see his survey~\cite{Smertnig-survey} and the recent article~\cite{Smertnig-BF}. 

A first \textit{unifying theory of factorizations} was proposed by Tringali~\cite{Salvo1} and further developed by Cossu and Tringali~\cite{Laura-Salvo}. It subsumes a wide variety of factorization phenomena all across mathematics by suggesting a common language, and it has been fruitfully applied to the classical problem of characterizing atomicity in monoids~\cite{Salvo2}.

Although we will stay within the realm of commutative cancellative monoids which is classical for factorization theory, our considerations on non-unique factorizations of elements of ultraproducts as products of irreducible elements are non-standard in two ways:
\begin{enumerate}
\item As Schoutens points out in the introduction of his inspiring book on the use of ultraproducts in commutative algebra~\cite{Schoutens}, these objects have been ``shunned perhaps because they were conceived as non-algebraic, belonging to the alien universe of set-theory and non-standard arithmetic, a universe in which most mathematicians did not, and still do not feel too comfortable''. 
\item Even when the monoids in the components of an ultraproduct have certain finiteness conditions, such as atomic or BF, typically assumed in factorization theory, the ultraproduct itself fails to satisfy them in all non-trivial situations.
\end{enumerate}

It is, however, evident that the ultraproduct construction has been very successfully applied to important problems in algebra such as the proof of Artin's Conjecture by Ax and Kochen~\cite{Ax-Kochen} and a recent non-standard approach to the Almost Purity Theorem and related results due to Jahnke and Kartas~\cite{Franzi}.

In view of factorization theory, these objects are useful because structural results on lengths of factorizations come in very naturally based on some knowledge of the component monoids, see Section~\ref{section:lengths}. Moreover, they connect objections that are model-theoretic in nature to classical factorizations and lead to classes of unexpected examples of monoids as we point out in Section~\ref{section:examples}. A new and unconventional view-point is taken in Section~\ref{section:infinite}, where we distinguish infinite cardinals as multiplicities of lengths of factorizations in ultraproducts and generalize the results of Sections~\ref{section:lengths} and~\ref{section:examples}, possibly with some additional set-theoretical assumptions such as the generalized continuum hypothesis or the existence of measurable cardinals.

The last two sections are attributed to applications to the classical theory of non-unique factorizations. We give a new proof, based on ultraproducts, for a theorem~\cite[Theorem 3.7]{application} of Geroldinger, Schmid and Zhong from 2017 that is related to additive combinatorics. These ideas originated in a personal discussion with Victor Fadinger-Held. 

A natural ingredient in this, though hidden at first sight, is the protoproduct construction as introduced by Schoutens~\cite[Chapter 12]{Schoutens} and successfully employed in algebraic topology~\cite{algtop,algtop1}. Based on it, we propose a general method of using ultraproducts for studying monoid arithmetic and give several possible fields of application. Finally, we prove that a protoproduct of Krull monoids is again a Krull monoid and we explicitely determine a divisor theory under customary assumptions.

Before we start our considerations, we want to recall some notions from factorization theory and model theory for the convenience of the reader.

\section{Preliminaries}

Throughout this paper, $\Lambda$ will always denote a (usually infinite) set. We make the following standing assumption:
\begin{center}
All algebraic structures are commutative.
\end{center}

\noindent
\textbf{Monoids.} We write semigroups multiplicatively. A semigroup $H$ is called \textit{cancellative} if $ac = bc$ implies $a = b$ for all $a,b,c \in H$.
A \textit{monoid} is a semigroup with identity element $1$. By $H^\times$ we denote the set of units, that is, invertible elements of the monoid $H$. If $D$ is an integral domain then the set $D^\bullet$ of all regular elements (all elements that do not divide $0$) of $D$ endowed with the ring multiplication is a cancellative monoid.

In the following we will define several notions for monoids. These definitions are transferred to a domain $D$ by considering them for the monoid $D^\bullet$ of regular elements.\\

\noindent
\textbf{Atoms and associates.}
A non-unit $a$ of a cancellative monoid $H$ is called an \textit{atom} or \textit{irreducible}, if $a = bc$ implies $b \in H^\times$ or $c \in H^\times$ for all $b,c \in H$. We denote by $\mathcal{A}(H)$ the set of all atoms of $H$. 

Elements $a,b \in H$ are called \textit{associated} if there exists $\varepsilon \in H^\times$ such that $a = \varepsilon b$. If $a$ and $b$ are associated, we write $a \sim b$.\\

\noindent
\textbf{Factorizations.} For a formal mathematical definition of factorizations and related concepts and results, see the monograph by Geroldinger and Halter-Koch on non-unique factorizations~\cite{GHK}. We use a description that is more informal but enough for the purpose of this paper. Let $H$ be a cancellative monoid and $a \in H$. A \textit{factorization} of $a$ is an expression of the form
\[
a = u_1 \cdots u_\ell,
\]
where $\ell \in \mathbb{N}_0$ and $u_1,\ldots,u_\ell \in \mathcal{A}(H)$. The integer $\ell$ is called the \textit{length} of the factorization.

The monoid $H$ is called \textit{atomic} if every non-unit of $H$ admits a factorization. 

Two factorizations $u_1 \cdots u_n$ and $v_1 \cdots v_m$ are called \textit{essentially the same} if $m = n$ and there exists a permutation $\sigma$ of the indices such that $u_i \sim v_{\sigma(i)}$ for all $i \in \{1,\ldots ,n\}$. Two factorizations that are not essentially the same are called \textit{essentially different}. Being essentially the same is clearly an equivalence relation on the set of all factorizations of an element in a cancellative monoid. 

We say that $H$ is \textit{factorial} if $H$ is atomic and every two factorizations of an element of $H$ are essentially the same.\\
%

\noindent
\textbf{Multisets.}  A \textit{multiset} $L$ is a pair $(M,\#_L)$, where $M$ is a set and $\#_L$ is a map defined on $M$ such that $\#_L(m)$ is a cardinal number for every $m \in M$. $L$ is called \textit{infinite} if $\{m \in M \mid \#_L(m) \neq 0 \}$ is infinite and it is called \textit{finite} otherwise. For $m \in M$, the cardinal $\#_{L}(m)$ is called the \textit{multiplicity} of $m$ in $L$. \\
If $L = (M,\#_L)$ is a multiset then $m \in M$ is called an \textit{element} of $L$ if $\#_L(m) \geq 1$. If $K=(M,\#_K)$ is another multiset, we write $L \subseteq K$ if every element of $L$ is an element of $K$, and $L = K$ if $\#_L(m) = \#_K(m)$ for every $m \in M$. \\

\noindent
\textbf{Sets of lengths.} The \textit{multiset of lengths} $L_H(h)= L(h)$ of an element $h$ in a monoid $H$ is the multiset $L(h) \subseteq \mathbb{N}_{0}$ such that $\#_{L(h)}(\ell)$  is the cardinality of the set of equivalence classes of factorizations with length $\ell$ of $h$ with respect to being essentially the same, for every $\ell \in \mathbb{N}_0$. The \textit{set of lengths} of $h$ is the underlying set of $L(h)$ which, by abuse of notation, we also denote by $L(h)$. We call 
\[
\mathcal{L}(H) = \{L(h) \mid h \in H\}
\]
the \textit{system of sets of lenghts} of $H$. We will use the notation $\mathcal{L}(H)$ only in Section~\ref{section:application} and, there, $L(h)$ will always denote a set.\\

\noindent
\textbf{Filters and ultrafilters.} A \textit{filter} on $\Lambda$ is a non-empty collection $\mathcal{U}$ of subsets of $\Lambda$ satisfying the following properties:
\begin{itemize}
\item[•] $\emptyset \notin \mathcal{U}$.
\item[•] If $M \subseteq \Lambda$ and $U \in \mathcal{U}$ with $U \subseteq M$, then $M \in \mathcal{U}$.
\item[•] If $U, V \in \mathcal{U}$, then $U \cap V \in \mathcal{U}$.
\end{itemize}
A filter $\mathcal{U}$ on $\Lambda$ is called an \textit{ultrafilter}, if $M \in \mathcal{U}$ or $\Lambda \setminus M \in \mathcal{U}$ for all $M \subseteq \Lambda$. \\
It is easy to see that, for an ultrafilter $\mathcal{U}$ on $\Lambda$ and $U,V \subseteq \Lambda$ with $U \cup V \in \mathcal{U}$, either $U \in \mathcal{U}$ or $V \in \mathcal{U}$. \\
An ultrafilter $\mathcal{U}$ on $\Lambda$ is called \textit{pricipal} or \textit{trivial} if for some $\lambda \in \Lambda$ it contains the singleton $\{\lambda \}$. Otherwise $\mathcal{U}$ is called \textit{non-principal} or \textit{non-trivial}. An ultrafilter $\mathcal{U}$ on $\Lambda$ is non-principal if and only if $\mathcal{U}$ contains no finite subset of $\Lambda$. It is a standard fact that every infinite set admits a non-trivial ultrafilter. \\

\noindent
\textbf{Ultraproducts.} All the concepts that we will recall now, can be defined in far greater generality, see, for instance, the standard textbook on model theory by Chang and Keisler~\cite{mod}. However, we only need the ultraproduct construction in the case of monoids. 

Let $\mathcal{U}$ be an ultrafilter on $\Lambda$. We define an equivalence relation $\equiv$ on the product $H = \prod_{\lambda \in \Lambda} H_\lambda$ of a given family of sets $H_\lambda$. For $r = (r_\lambda), s = (s_\lambda) \in H$ let
\begin{align*}
r \equiv s :\Leftrightarrow \{\lambda \in \Lambda \mid r_\lambda = s_\lambda \} \in \mathcal{U}.
\end{align*}
The \textit{ultraproduct} $H^* = \prod_{\lambda \in \Lambda}^\mathcal{U} H_\lambda$ of the $H_\lambda$ with respect to $\mathcal{U}$ is the set $H/_\equiv$. For an element $a = (a_\lambda) \in H$, the corresponding equivalence class in $H^*$ is denoted by $a^* = (a_\lambda)^*_{\lambda \in \Lambda} = (a_\lambda)^*$. Sometimes, if there is an additional index, we write $a^{(\lambda)}$ instead of $a_\lambda$ for the components. When all the $H_\lambda$ are monoids, the set $H^*$ endowed with the operation defined by $a^* \cdot b^* = (a_\lambda \cdot b_\lambda)^*$ is also a monoid.\\

\noindent
\textbf{First-order sentences.} An important property of ultraproducts is that they preserve first-order statements.  Roughly speaking, a first-order sentence in the language of monoids is a formula (without free variables) only using $=$, $\cdot$, $1$ and logical symbols such as quantifiers and sentential connectives, in such a way that quantification is only done over variables for the elements of the monoid.

We will make use of the following fundamental theorem for ultraproducts~\cite[Theorem 4.1.9]{mod}:

\begin{theorem*}[\L o\'s]\label{Los}
A first order sentence $\varphi$ is satisfied by $H^*$ if and only if the set of all $\lambda \in \Lambda$ such that $H_\lambda$ satisfies $\varphi$ is in $\mathcal{U}$.
\end{theorem*}

It is easy to see that being cancellative is a first order sentence in the language of monoids and therefore the ultraproduct $H^*$ of cancellative monoids $H_\lambda$ is again cancellative. As we will see in Section~\ref{section:examples}, many useful properties that are common in the theory of non-unique factorizations in monoids cannot be expressed as first-order statements in the language of monoids. 

We want to note that the very basic Lemmas~\ref{2.1} and~\ref{2.2} can also be deduced from a more general version of the Theorem of \L o\'s that deals not only with sentences but with general first-order formulas. However, since this paper is intended rather to reach algebraists interested in factorizations, we spell out the easy proofs that might be quite instructive when first working with ultraproducts.\\

\section{Sets of lengths in ultraproducts}\label{section:lengths}

Let $\mathcal{U}$ be an ultrafilter on the set $\Lambda$ and let $H^* = \prod_{\lambda \in \Lambda}^\mathcal{U} H_\lambda$ be the ultraproduct of the cancellative monoids $H_\lambda$ with respect to $\mathcal{U}$. If $\Lambda$ is finite, then $\mathcal{U}$ is a principal ultrafilter and therefore $H^* \cong H_\lambda$ for some $\lambda \in \Lambda$. To exclude this trivial case, let $\Lambda$ be infinite and let $\mathcal{U}$ be non-principal.

\begin{lemma}\label{2.1}
For $a \in H = \prod H_\lambda$, the corresponding element $a^*$ is a unit in $H^*$ if and only if $\{ \lambda \in \Lambda \mid a_\lambda \in H_\lambda^\times\} \in \mathcal{U}$. 
\end{lemma}

\begin{proof}
If $a^*$ is a unit in $H^*$, let $b \in H$ such that $a^*b^* = 1$. Then $\{ \lambda \in \Lambda \mid a_\lambda \in H_\lambda^\times\} \supseteq \{ \lambda \in \Lambda \mid a_\lambda b_\lambda = 1\} \in \mathcal{U}$. 

Now let $U = \{ \lambda \in \Lambda \mid a_\lambda \in H_\lambda^\times\} \in \mathcal{U}$. For $\lambda \in U$, let $b_\lambda = a_\lambda^{-1}$. Let $b_\lambda = 1$ (or any other arbitrary element of $H_\lambda$) if $\lambda \in \Lambda \setminus U$. Then $a^*b^* = 1$.
\end{proof}

\begin{lemma}\label{2.2}
Let $r \in H = \prod H_\lambda$. Then $r^* \in \mathcal{A}(H^*)$ if and only if $\{ \lambda \in \Lambda \mid r_\lambda \in \mathcal{A}(H_\lambda) \} \in \mathcal{U}$.
\end{lemma}

\begin{proof}
Let $U = \{ \lambda \in \Lambda \mid r_\lambda \in \mathcal{A}(H_\lambda) \}$. If $U \in \mathcal{U}$ and $a, b \in H$ such that $r^* = a^* b^*$ then $V = U \cap \{\lambda \in \Lambda \mid r_\lambda = a_\lambda b_\lambda \} \in \mathcal{U}$. Since $r_\lambda$ is an atom for all $\lambda \in U$, it follows that $V \subseteq W_a \cup W_b$, where $W_s = \{ \lambda \in \Lambda \mid s_\lambda \in H_\lambda^\times \}$ for $s \in \{ a,b \}$. Since $\mathcal{U}$ is an ultrafilter, it follows that $W_a \in \mathcal{U}$ or $W_b \in \mathcal{U}$. By Lemma~\ref{2.1}, $a^*$ is a unit or $b^*$ is a unit. Therefore, $r^* \in \mathcal{A}(H^*)$.

Conversely, if $\Lambda \setminus U \in \mathcal{U}$, for every $\lambda \in \Lambda \setminus U$, let $a_\lambda, b_\lambda \in H_\lambda$ be non-units such that $r_\lambda = a_\lambda b_\lambda$. If $\lambda \in U$, let $a_\lambda = b_\lambda = 1$ (or any other element of $H_\lambda$). Then $a^*$ and $b^*$ are non-units such that $r^* = a^*b^*$ and, therefore, $r^*$ is not an atom.
\end{proof}

\begin{remark}\label{2.3} \phantom{}
\begin{itemize}
\item[(1)] If $U = \{\lambda \in \Lambda \mid H_\lambda \text{ is non-atomic} \} \in \mathcal{U}$, then $H^*$ is non-atomic. Indeed, let $a_\lambda$ be a non-zero non-unit of $H_\lambda$ without factorization into atoms, for $\lambda \in U$, and let $a_\lambda = 1$ for $\lambda \in \Lambda \setminus U$. Assume to the contrary that $a^* = u_1^* \cdots u_n^*$ with $u_1,\ldots,u_n \in H$ such that $u_1^*,\ldots,u_n^* \in \mathcal{A}(H^*)$. Then there exists some $V \in \mathcal{U}$ such that $V \subseteq U$ and, for all $\lambda \in V$, we have $a_\lambda = u_1^{(\lambda)}\cdots u_n^{(\lambda)}$ and $u_1^{(\lambda)},\ldots,u_n^{(\lambda)} \in \mathcal{A}(H_\lambda)$. In particular, $V \neq \emptyset$. Let $\lambda \in V$. Then $a_\lambda = u_1^{(\lambda)} \cdots u_n^{(\lambda)}$ is a factorization into atoms, which is a contradiction.
\item[(2)] The converse to (1) is not true. What is more, the non-trivial ultraproduct $D^*$ of factorial domains $D_\lambda$ is never atomic. As an illustration, consider $\Lambda = \mathbb{N}$, and for every $\lambda \in \mathbb{N}$, let $D_\lambda$ be a factorial domain.

The non-trivial ultrafilter $\mathcal{U}$ on $\mathbb{N}$ does not contain finite sets. Therefore every $U \in \mathcal{U}$ is an unbounded subset of $\mathbb{N}$. Let $p_\lambda \in D_\lambda$ be a prime element and for every $\lambda \in \mathbb{N}$, define $r_\lambda = p_\lambda^\lambda$. We claim that $r^*$ has no factorization into atoms of $D^*$.

Assume to the contrary that $r^* = u_1^* \cdots u_n^*$ with $u_i \in D = \prod D_\lambda$ and $u_i^* \in \mathcal{A}(D^*)$ for all $i \in \{1,\ldots,n\}$. Let 
\begin{align*}
U = \{\lambda \in \Lambda \mid r_\lambda = u_1^{(\lambda)}\cdots u_n^{(\lambda)} \} \cap \bigcap_{i = 1}^n \{\lambda \in \Lambda \mid u_i^{(\lambda)} \in \mathcal{A}(D_\lambda)\}.
\end{align*}
Clearly, $U$ is an element of $\mathcal{U}$ and is therefore unbounded in $\mathbb{N}$. Let $\lambda \in U$ with $\lambda>n$. Then $p_\lambda^\lambda = r_\lambda = u_1^{(\lambda)} \cdots u_n^{(\lambda)}$, which is a contradiction to our assumption that $D_\lambda$ be factorial.  
\end{itemize}
\end{remark}

\begin{remark}
The mechanism explored in Remark~\ref{2.3} reminds of the ``blow-up phenomenon'', named so by Cossu and Tringali~\cite{Laura-Salvo}, but somehow in a reversed way. It seems to be intuitive that the element $(p_1^1,p_2^2,\ldots,p_\lambda^\lambda,\ldots)^*$ in our example has a factorization; it should be a prime power which it actually is in a non-standard universe of set-theory. But there is no obvious way to associate this idea with classical factorizations. It might be fertile to employ a variant of the novel ideas of~\cite{Laura-Salvo}.
\end{remark}

Since ultraproducts are rarely atomic, it makes sense for us to only consider elements of $H^*$ which factor as products of atoms. We denote by $H^*_A$ the set of all elements of $H^*$ which have some factorization. Note that $H^*_A$ is the submonoid of $H^*$ generated by units and atoms of $H^*$. In particular,
\begin{align*}
H^*_A = \{a^* \mid a \in H \text{ and } \exists U \in \mathcal{U} \ \exists N \in \mathbb{N} \ \forall \lambda \in U \ \{0,\ldots ,N\} \cap L(a_\lambda) \neq \emptyset \}
\end{align*}
is the set of all elements of $H^*$ which have some representative $a \in H$ such that there is $N \in \mathbb{N}$ such that for each $\lambda$ in some set of $\mathcal{U}$ there exists a factorization of $a_\lambda$ in at most $N$ atoms. \\

For the remainder of this section, if $L \subseteq \mathbb{N}_{\geq 2}$ is a multiset and $\ell \in \mathbb{N}_{\geq 2}$, we write $\#_L(\ell) = \infty$ if and only if $\#_L(\ell)$ is an infinite cardinal.
The next theorem gives a full description of sets of lengths $L(r^*)$ of elements $r^*$ of ultraproducts up to the distinction of infinite multiplicities. We deal with infinite multiplicities of lengths of factorizations in Section~\ref{section:infinite} separately in order to get a smooth statement (without additional set-theoretical assumptions) in the more classical context of the present section that might be interesting to a broader audience.

\begin{theorem}\label{2.4}
Let $(H_\lambda)_{\lambda \in \Lambda}$ be a family of commutative cancellative monoids. Denote $H = \prod_{\lambda \in \Lambda} H_\lambda$ and let $H^* = \prod_{\lambda \in \Lambda}^\mathcal{U} H_\lambda$ be the ultraproduct with respect to an ultrafilter $\mathcal{U}$ on $\Lambda$. \\
Let $r \in H$, $\ell \in \mathbb{N}_{0}$ and $n \in \mathbb{N}_0$.
\begin{itemize}
\item[(1)] $\#_{L(r^*)}(\ell) = n$ if and only if $\{\lambda \in \Lambda \mid \#_{L(r_\lambda)} (\ell)= n \} \in \mathcal{U}$.
\item[(2)] $\#_{L(r^*)}(\ell) = \infty$ if and only if $\{\lambda \in \Lambda \mid \#_{L(r_\lambda)} (\ell) > m \} \in \mathcal{U}$ for all $m \in \mathbb{N}$.
\end{itemize}
\end{theorem}

\begin{proof}
To show (1), first let $\#_{L(r^*)}(\ell) = n$. Then we can pick exactly $n$ essentially different factorizations of $r^*$ into $\ell$ atoms
\begin{align*}
r^* &= u_{1,1}^*\cdots u_{1,\ell}^* \\
&\vdots \\
r^* &= u_{n,1}^* \cdots u_{n,\ell}^*.
\end{align*}
Let $U \in \mathcal{U}$ such that for all $\lambda \in U$ the following hold:
\begin{itemize}
\item[(i)] $r_\lambda = u_{i,1}^{(\lambda)}\cdots u_{i,\ell}^{(\lambda)}$ for all $i \in \{1,\ldots,n\}$.
\item[(ii)] $u_{i,j}^{(\lambda)}$ is an atom for all $i \in \{1,\ldots,n\}$ and $j \in \{1,\ldots,\ell\}$.
\end{itemize}
Let $V$ be the set of all $\lambda \in U$ such that the factorizations $u_{i,1}^{(\lambda)}\cdots u_{i,\ell}^{(\lambda)}$ are pairwise essentially different. Then $U_{\geq} = \{\lambda \in \Lambda \mid \#_{L(r_\lambda)} (\ell) \geq n \} \supseteq V \in \mathcal{U}$. Therefore, either $\{ \lambda \in \Lambda \mid \#_{L(r_\lambda)}(l) = n \}  \in \mathcal{U}$ or $\{ \lambda \in \Lambda \mid \#_{L(r_\lambda)}(l) > n \}  \in \mathcal{U}$. Assume to the contrary that $\{ \lambda \in \Lambda \mid \#_{L(r_\lambda)}(l) > n \}  \in \mathcal{U}$. Then we can construct at least $n+1$ pairwise essentially different factorization of $r^*$ into atoms, which is a contradiction.

Conversely, assume that $U = \{\lambda \in \Lambda \mid \#_{L(r_\lambda)} (\ell)= n \} \in \mathcal{U}$. Then we can immediately construct $n$ essentially different factorizations of $r^*$ into $\ell$ atoms of $H^*$ and every factorization or $r^*$ of length $\ell$ is essentially the same as one of these.

For the proof of (2), first assume that $\{\lambda \in \Lambda \mid \#_{L(r_\lambda)} (\ell) > m \} \in \mathcal{U}$ for all $m \in \mathbb{N}$. Then we can, in particular, construct $m$ essentially different factorizations of $r^*$ into $\ell$ atoms of $H^*$ for every $m \in \mathbb{N}$. Therefore $\#_{L(r^*)}(\ell) = \infty$.

Conversely, assume that there exists some $m \in \mathbb{N}$ such that 
\begin{align*}
\bigcup_{n=0}^m \{\lambda \in \Lambda \mid \#_{L(r_\lambda)} (\ell) =n \} =  \{\lambda \in \Lambda \mid \#_{L(r_\lambda)} (\ell) \leq m \} \in \mathcal{U}.
\end{align*}
Then, since $\mathcal{U}$ is an ultrafilter, there exists some $n \in \{0,\ldots ,m\}$ such that $\{\lambda \in \Lambda \mid \#_{L(r_\lambda)} (\ell) =n \} \in \mathcal{U}$. By (1), it follows that $\#_{L(r^*)}(\ell) = n < \infty$.
\end{proof}

A cancellative monoid $H$ is \textit{half-factorial} if it is atomic and the set of lengths $L(r)$ is a singleton for every non-unit $r \in H$. Half-factoriality plays an interesting role in the study of class groups of Dedekind domains and, in particular, rings of integers of number fields. By a famous result of Carlitz~\cite{Carlitz-halffactorial}, a ring of integers in a number field has class number at most $2$ if and only if it is half-factorial. Geroldinger and Göbel~\cite{Geroldinger-Goebel} proved, for large classes of Abelian groups, that they appear as the class group of a half-factorial Dedekind domain. However, for general Abelian groups, this realization problem is still open.

\begin{corollary}\label{2.5}
Let $(H_\lambda)_{\lambda \in \Lambda}$ be a family of commutative cancellative monoids, $H = \prod_{\lambda \in \Lambda} H_\lambda$ and let $H^* = \prod_{\lambda \in \Lambda}^\mathcal{U} H_\lambda$ be the ultraproduct with respect to some ultrafilter $\mathcal{U}$ on $\Lambda$. \\
Let $r \in H$ and $\ell \in \mathbb{N}_{0}$.
\begin{itemize}
\item[(1)] $\ell \in L(r^*)$ if and only if $\{\lambda \in \Lambda \mid \ell \in L(r_\lambda) \} \in \mathcal{U}$.
\item[(2)] If $\{\lambda \in \Lambda \mid H_\lambda \text{ is half-factorial} \}$ then the maximal atomic submonoid
\begin{align*}
H^*_A = \{r^* \mid r \in H \text{ and } \exists U \in \mathcal{U} \ \exists N \in \mathbb{N} \ \forall \lambda \in \Lambda \ \{0,\ldots ,N\} \cap L(r_\lambda) \neq \emptyset \}
\end{align*}
of $H^*$ is half-factorial.
\end{itemize}
\end{corollary}

\begin{proof}
(1) is a direct consequence of Theorem \ref{2.4} by disregarding the multiplicity of $\ell$ in the multisets of lengths. \\
(2) follows immediately from (1).
\end{proof}

It is clear from the previous corollary that, in particular, if $\{\lambda \in \Lambda \mid H_\lambda \text{ is factorial} \} \in \mathcal{U}$ then $H^*_A$ is half-factorial. We will see now that even more is true.

\begin{proposition}
If $\{\lambda \in \Lambda \mid H_\lambda \text{ is factorial} \} \in \mathcal{U}$ then $H^*_A$ is a factorial submonoid of $H^*$.
\end{proposition}

\begin{proof}
We know by Corollary \ref{2.5} that $H^*_A$ is half-factorial. So, it suffices to consider factorizations of equal length. Let $u_1,\ldots ,u_n,v_1,\ldots ,v_n \in H$ such that $u_i^*,v_i^* \in \mathcal{A}(H^*)$ for all $i \in \{1,\ldots ,n\}$ and $u_1^*\cdots u_n^* = v_1^*\cdots v_n^*$. There exists some $U \in \mathcal{U}$ such that for all $\lambda \in U$ we have $u_i^{(\lambda)}, v_i^{(\lambda)} \in \mathcal{A}(H_\lambda)$, $H_\lambda$ is factorial and $u_1^{(\lambda)}\cdots u_n^{(\lambda)} = v_1^{(\lambda)}\cdots v_n^{(\lambda)}$. Therefore, for every $\lambda \in U$ there exists some permutation $\sigma$ such that $u_i^{(\lambda)} \sim v_{\sigma(i)}^{(\lambda)}$ for all $i \in \{1,\ldots ,n\}$. Write the symmetric group on $\{1,\ldots,n\}$ as $S_n = \{ \sigma_1,\ldots ,\sigma_k\}$, where $k = n!$, and for every $r \in \{1,\ldots ,k\}$ let 
\begin{align*}
U_r = \{ \lambda \in U \mid \forall i \in \{ 1,\ldots ,n \} \ u_i^{(\lambda)} \sim v_{\sigma_r(i)}^{(\lambda)} \}.
\end{align*}
Then $U = U_1 \cup \ldots  \cup U_k$ and, therefore, there exists some $r \in \{1,\ldots ,k\}$ such that $U_r \in \mathcal{U}$. Set $\sigma = \sigma_r$. Then $u_i^{(\lambda)} \sim v_{\sigma(i)}^{(\lambda)}$ for all $\lambda \in U_r$ and all $i \in \{1,\ldots ,n\}$. It follows that $u_i^* \sim v_{\sigma(i)}^*$ for all $i \in \{1,\ldots ,n\}$. This shows that $H^*_A$ is factorial.
\end{proof}

\section{First-order properties for non-unique factorizations and monster examples}\label{section:examples}

We now consider several properties that are standard and useful in the theory of non-unique factorizations. As will turn out, none of them can be formulated as first-order statement in the language of monoids. Since it is not the place to recall all these notions here, we want to refer to the standard reference~\cite{GHK}.

%
%
%
%

\begin{proposition}
The following properties of commutative cancellative monoids cannot be formulated as first-order statements in the language of monoids:
\begin{enumerate}
\item atomic
\item bounded factorizations (BF)
\item finite factorizations (FF)
\item half-factorial
\item factorial
\item $v$-Noetherian
\item Krull
\end{enumerate}
\end{proposition}

\begin{proof}
For every property $P$ in the list, factorial $\Rightarrow$ $P$ $\Rightarrow$ atomic. So, by the Theorem of \L o\'s, it suffices to construct an ultraproduct of factorial monoids that is not atomic. This is Remark~\ref{2.3}(2).
\end{proof}

Our goal for this section is to construct an integral domain that realizes every multiset as a multiset of lengths of some element. Analogous examples have been known in the setting of finite multisets, see Example~\ref{3.4}. To the best of our knowledge, we are the first to give an example in the infinite setting. 

In the following, to avoid confusion, we use the symbol $\omega$ for the set of all non-negative integers as an index set while the notation $\mathbb{N}_0$ is used for sets of lengths.

\begin{theorem} \label{3.3}
Let $(H_\lambda)_{\lambda \in \omega}$ be a family of cancellative monoids such that for all $\lambda \in \omega$ the following property holds:

(F) Every finite multiset $L \subseteq \mathbb{N}_{\geq 2}$ is the multiset of lengths of some element in $H_\lambda$.

Let $\mathcal{U}$ be a non-principal ultrafilter on $\omega$.
Then every multiset $L \subseteq \mathbb{N}_{\geq 2}$ is the multiset of lengths of some element $r^* \in H^* = \prod_{\lambda \in \omega}^\mathcal{U} H_\lambda$, in the sense that for all $\ell \in \mathbb{N}_{\geq 2}$
\begin{itemize}
\item[(i)] $\#_{L(r^*)}(\ell) = \#_L(\ell)$ if $\#_L(\ell) < \infty$ and
\item[(ii)] $\#_{L(r^*)}(\ell) = \infty$ if and only if $ \#_L(\ell) = \infty$.
\end{itemize}
\end{theorem}

\begin{proof}
Let $L \subseteq \mathbb{N}_{\geq 2}$ be a multiset. Since we do not distinguish infinite cardinalities, we can assume that $\#_L( \ell)$ is countable for every $\ell \in \mathbb{N}_{\geq 2}$. Then $L$ is countable (as a countable union of countable multisets containing one single with countable multiplicity). If $L$ is finite, we are done by picking for every $\lambda \in \Lambda$ an element $r_\lambda \in H_\lambda$ with set of lengths $L$ and setting $r^* = (r_\lambda)^*$. If $|L| = \aleph_0$ then there exists a bijection $f: \omega \to L$ (i.e. every appearance of some element of $L$ gets assigned to exactly one number in $\omega$). Using Property (F), for every $\lambda \in \omega$, pick $r_\lambda \in H_\lambda$ such that $L(r_\lambda) = \{f(0),\ldots ,f(\lambda)\}$. We claim that $L(r^*) = L$. For this, it suffices to show that $\#_{L(r^*)}(\ell) = \#_{L}(\ell)$ for all $\ell \in \mathbb{N}_{\geq 2}$. So let $\ell \in \mathbb{N}_{\geq 2}$.

\textbf{Case 1:} Let $\#_{L}(\ell) = \infty$. By Corollary~\ref{2.5}, we have to show for every $m \in \mathbb{N}$ that $\{\lambda \in \omega \mid \#_{L(r_\lambda)}(\ell) > m \} \in \mathcal{U}$. Let $m \in \mathbb{N}$. Then there exists some $\lambda \in \omega$ such that the multiplicity of $\ell$ in $L(r_\lambda) = \{f(0),\ldots ,f(\lambda)\}$ is bigger than $m$, because $\#_{L}(\ell) = \infty$ and $f$ is bijective. This is clearly also true for all $\mu \in \omega$ with $\mu \geq \lambda$. This completes the proof of Case 1 because $\{\mu \in \omega \mid \mu \geq \lambda \} \in \mathcal{U}$ which follows from the fact that its complement in $\omega$ is finite.

\textbf{Case 2:} Let $\#_{L}(\ell) = n < \infty$. We have to show that $\{\lambda \in \omega \mid \#_{L(r_\lambda)} (\ell)= n \} \in \mathcal{U}$. Since the multiplicity of $\ell$ in $L$ is $n$ and $f$ is bijective, there exists some $\lambda \in \omega$ such that, for all $\mu \geq \lambda$, the multiplicity of $\ell$ in $\{f(0),\ldots ,f(\mu)\} = L(r_\mu)$ is $n$. Since $\omega \setminus \{\mu \in \omega \mid \mu \geq \lambda\} = \{0,\ldots ,\lambda - 1 \}$ is finite, it follows that $\{\lambda \in \Lambda \mid \#_{L(r_\lambda)} (\ell)= n \} \in \mathcal{U}$.
\end{proof}

To complete our example, we have to present classes of commutative cancellative monoids satisfying property (F). This can be done in different ways.

\begin{example} \label{3.4}
\begin{itemize}
\item[(1)] Let $G$ be an infinite Abelian group and $\mathcal{B}(G)$ the monoid of zero-sum sequences over $G$, see Section~\ref{section:application}. It is well-known~\cite[Section 3.4]{GHK} that $\mathcal{B}(G)$ is a Krull monoid with divisor class group $G$ and prime divisors in every class. \\
Let $K$ be a field. Then, by results of Gilmer~\cite{Gilmer},  and of Fadinger-Held and the present author~\cite{KrullPrimeDiv} the semigroup ring $K[\mathcal{B}(G)]$ is a Krull domain with divisor class group $G$ and prime divisors in every class. By a theorem of Kainrath~\cite{Kainrath}, every finite subset of $\mathbb{N}_{\geq 2}$ can be realized as set of lengths of some element in $\mathcal{B}(G)$ and, therefore, also of some element of $K[\mathcal{B}(G)]$. This result was extended to finite multisets for large classes of groups $G$ in~\cite[Theorem 7.4.1]{GHK}.
\item[(2)] For an integral domain $D$ with quotient field $K$, we denote by $\text{Int}(D) = \{ f \in K[X] \mid f(D) \subseteq D \}$ the ring of integer-valued polynomials over $D$. It is shown in~\cite{Frisch} that, for a Dedekind domain $D$ with infinitely many maximal ideals which are all of finite index (e.g. $D = \mathbb{Z}$), every finite multiset in $\mathbb{N}_{\geq 2}$ can be realized as multiset of lengths of some element in $\text{Int}(D)$. The analogous result for discrete valuation rings $D$ was proven in~\cite{FFW} and~\cite{Int-Krull-Prime}.
\end{itemize}
\end{example}

\begin{example}
Let $H^*$ be a monoid as in Theorem~\ref{3.3} which exists and can be chosen torsion-free (see~\cite{Gilmer_Semigroup-Rings}) by Example~\ref{3.4}. Let $K$ be a field. Then the semigroup ring $K[H^*]$ is an integral domain such that every multiset of integers greater than $2$ is the multiset of lengths of an element in $K[H^*]$ (in the sense of Theorem~\ref{3.3}).
\end{example}

\section{The distinction of infinite cardinals}\label{section:infinite}

We now want to look beyond the investigations of Sections~\ref{section:lengths} and~\ref{section:examples} and also consider different infinite cardinalities as multiplicities in sets of lengths. We are able to give partial generalizations of Theorem~\ref{2.4}. Using these, we generalize Theorem~\ref{3.3} (where we used $\aleph_0$ as a bound for the index set and the multiplicities in sets of lengths of the component monoids $H_\lambda$) from $\aleph_0$ to arbitrary measureable cardinals.

Again, let $(H_\lambda)_{\lambda \in \Lambda}$ be a family of commutative cancellative monoids, $\mathcal{U}$ an ultrafilter on $\Lambda$ and $H^*$ (respectively $H$) the ultraproduct with respect to $\mathcal{U}$ (respectively the product) of the $H_\lambda$. 

If $\kappa = \aleph_\alpha$ is a cardinal then we denote by $\kappa^+ = \aleph_{\alpha +1}$ the successor cardinal of $\kappa$. 

The next proposition lists the two partial statements of Theorem \ref{2.4} that can be extended without any further restrictions.

\begin{proposition} \label{4.1}
Let $r \in H = \prod_{\lambda \in \Lambda} H_\lambda$, $\ell \in \mathbb{N}_{\geq 2}$ and $\kappa$ a cardinal. 
\begin{itemize}
\item[(1)] If $\#_{L(r^*)}(\ell) \leq \kappa$ then $\{ \lambda \in \Lambda \mid \#_{L(r_\lambda)}(\ell) \leq \kappa \} \in \mathcal{U}$.
\item[(2)] If $\{ \lambda \in \Lambda \mid \#_{L(r_\lambda)}(\ell) > \mu \} \in \mathcal{U}$ for all cardinals $\mu < \kappa$ then $\#_{L(r^*)}(\ell) \geq \kappa$.
\end{itemize}
\end{proposition}

\begin{proof}
To show (1), note that, for $U = \{ \lambda \in \Lambda \mid \#_{L(r_\lambda)}(\ell) > \kappa \} \in \mathcal{U}$, we can construct at least $\kappa^+$ many essentially different factorizations of $r_\lambda$ of length $\ell$ for every $\lambda \in U$. These lead to at least $\kappa^+$ many essentially different factorizations of $r^*$ into $\ell$ atoms. \\
Now, for the proof of (2), assume that $\{ \lambda \in \Lambda \mid \#_{L(r_\lambda)}(\ell) > \mu \} \in \mathcal{U}$ for all cardinals $\mu < \kappa$ and assume to the contrary that $\mu := \#_{L(r^*)}(\ell) < \kappa$. Then, by assumption, we have that $\{ \lambda \in \Lambda \mid \#_{L(r_\lambda)}(\ell) > \mu \} \in \mathcal{U}$. It follows by (1) that $\mu = \#_{L(r^*)}(\ell) > \mu$ which is a contradiction.
\end{proof}

To give further generalizations of the statements in Section~\ref{section:lengths}, we have to impose set theoretical assumptions. Recall that the general continuum hypothesis, for short \textbf{GCH}, is the statement that $2^\kappa = \kappa^+$, for all infinite cardinals $\kappa$. \textbf{GCH} is independent of Zermelo-Fraenkel set theory with axiom of choice (\textbf{ZFC}). For a cardinal $\kappa$, the cofinality $\text{cof}(\kappa)$ of $\kappa$ is the least cardinal $\mu$ such that there exists an unbounded map $f: \mu \to \kappa$, that is, for all $\alpha \in \kappa$ there exists some $\beta \in \mu$ such that $f(\beta) \geq \alpha$.\\
We will use the following result from basic set theory:

\begin{remark} \label{4.2} \cite[Chapter IX, Satz 4.7(b)(ii)]{Ebbinghaus}
Assume that \textbf{GCH} holds. Let $\kappa$ be an infinite cardinal and $\mu$ be a cardinal such that $0 < \mu < \text{cof}(\kappa)$. Then $\kappa^\mu = \kappa$.
\end{remark}

\begin{proposition}\label{4.3}
Assume that \textbf{GCH} holds. Let $r \in H = \prod_{\lambda \in \Lambda} H_\lambda$, $\ell \in \mathbb{N}_{\geq 2}$ and $\kappa$ be a cardinal such that $|\Lambda| < \text{cof}(\kappa)$.
\begin{itemize}
\item[(1)] If $\{ \lambda \in \Lambda \mid \#_{L(r_\lambda)}(\ell) \leq \kappa \} \in \mathcal{U}$ then $\#_{L(r^*)}(\ell) \leq \kappa$.
\item[(2)] If $\{ \lambda \in \Lambda \mid \#_{L(r_\lambda)}(\ell) \leq \kappa \} \in \mathcal{U}$ and $\{ \lambda \in \Lambda \mid \#_{L(r_\lambda)}(\ell) > \mu \} \in \mathcal{U}$ for all cardinals $\mu < \kappa$ then $\#_{L(r^*)}(\ell) = \kappa$.
\end{itemize}
\end{proposition}

\begin{proof}
(2) follows immediately from (1) together with Proposition \ref{4.1}(2).

To show (1), let $U= \{ \lambda \in \Lambda \mid \#_{L(r_\lambda)}(\ell) \leq \kappa \} \in \mathcal{U}$ and for all $\lambda \in U$, let $\mathfrak{Z}_\lambda$ be a set of at most $\kappa$ essentially different factorizations with length $\ell$ of $r_\lambda$ such that every given factorization with length $\ell$ is essentially the same as an element in $\mathfrak{Z}_\lambda$. Moreover, for $\lambda \in \Lambda \setminus U$, let $\mathfrak{Z}_\lambda = \{1\}$. Now set
\begin{align*}
\mathfrak{Z}^* = \{ (Z_\lambda)_{\lambda \in \Lambda}^* \mid \forall \lambda \in \Lambda \ Z_\lambda \in \mathfrak{Z}_\lambda \}.
\end{align*}
We show that every factorization of $r^*$ with length $\ell$ is essentially the same as an element of $\mathfrak{Z}^*$. Let $Z^* = (Z_\lambda)^*$ be a factorization of $r^*$ with length $\ell$ and let $V \subseteq U$ such that $Z_\lambda$ is a factorization of $r_\lambda$ for every $\lambda \in V$. Then, for every $\lambda \in V$, we have that $Z_\lambda \in \mathfrak{Z}_\lambda$. We conclude that $Z^* \in \mathfrak{Z}^*$. 

To complete the proof, we note that $|\mathfrak{Z}^*| \leq \kappa^{|U|} \leq \kappa^{|\Lambda|} = \kappa$ by Remark \ref{4.2}.
\end{proof}

The next proposition completes our series of generalizations of Theorem \ref{2.4}. Here we need an additional assumption on the ultrafilter $\mathcal{U}$. Let $\kappa$ be a cardinal. The ultrafilter $\mathcal{U}$ on $\Lambda$ is called $\kappa$-\textit{complete} if for every family $(U_i)_{i \in I}$ of sets $U_i \in \mathcal{U}$ with $|I| < \kappa$ we have that $\bigcap_{i \in I} U_i \in \mathcal{U}$.

Principal ultrafilters are $\kappa$-complete for every cardinal $\kappa$. Every ultrafilter is $\aleph_0$-complete per definition. It is necessary for $\mathcal{U}$ being a $\kappa$-complete non-principal ultrafilter on $\Lambda$ that $\kappa \leq |\Lambda|$~\cite[Proposition 4.2.2]{mod}.

\begin{proposition} \label{4.4}
Let $r \in H = \prod_{\lambda \in \Lambda} H_\lambda$, $\ell \in \mathbb{N}_{\geq 2}$ and $\kappa$ be a cardinal such that $\mathcal{U}$ is $\kappa^+$-complete. If $\#_{L(r^*)}(\ell) \geq \kappa$ then $\{ \lambda \in \Lambda \mid \#_{L(r_\lambda)}(\ell) > \mu \} \in \mathcal{U}$ for all cardinals $\mu < \kappa$.
\end{proposition}

\begin{proof}
Let $\mathfrak{Z}^*$ be a set of $\kappa$ essentially different factorizations with length $\ell$ of $r^*$. For every $Z^* = (Z_\lambda)^*_\lambda$ there exists $U_{Z^*} \in \mathcal{U}$ such that the $Z_\lambda$ are $\kappa$ essentially different factorizations with length $\ell$ of $r_\lambda$ for every $\lambda \in U_{Z^*}$. Since $\mathcal{U}$ is $\kappa^+$-complete, it follows that $U = \bigcap_{Z^* \in \mathfrak{Z}^*} U_{Z^*} \in \mathcal{U}$. Let $U' \subseteq U$ with $U' \in \mathcal{U}$ such that for all $Y^*, Z^* \in \mathfrak{Z}^*$ and for all $\lambda \in U'$ we have that $Y_\lambda$ and $Z_\lambda$ are essentially different. 

Assume to the contrary that there exists a cardinal $\mu < \kappa$ such that $V = \{ \lambda \in \Lambda \mid \#_{L(r_\lambda)}(\ell) \leq \mu \} \in \mathcal{U}$. In particular, $U' \cap V \neq \emptyset$ and we can pick $\lambda \in U' \cap V$. On the one hand, $\lambda \in U'$ and therefore $\#_{L(r_\lambda)}(\ell) \geq \kappa$. But, on the other hand, $\lambda \in V$ and therefore $ \#_{L(r_\lambda)}(\ell) \leq \mu < \kappa$, which is a contradiction.
\end{proof}

A cardinal $\kappa$ is called \textit{measurable} if there exists a non-principal $\kappa$-complete ultrafilter on $\kappa$. So if $\Lambda$ is a measureable cardinal and $\kappa < \Lambda$, then Proposition~\ref{4.4} applies.

It is consistent with \textbf{ZF} (Zermelo-Fraenkel set theory without axiom of choice) that there exists a cardinal $\kappa$ such that $\kappa^+$ is measurable and it is consistent with \textbf{ZFC} that there exists an uncountable measureable cardinal. On the other hand, the existence of an uncountable measurable cardinal cannot be proven based on \textbf{ZFC} since every uncountable measurable cardinal must be inaccessible. The cardinal $\aleph_0$ clearly is measurable. For further information on this topic, see~\cite{Jech}.

The following is a very basic and easy result on $\kappa$-complete ultrafilters whose proof is included here for the convenience of the reader.

\begin{remark} \label{4.5}
Let $\mathcal{U}$ be a non-principal $\kappa$-complete ultrafilter on $\Lambda$ for some cardinal $\kappa$. Then $\mathcal{U}$ contains no set of cardinality less than $\kappa$.
\end{remark}

\begin{proof}
Assume to the contrary that $U \in \mathcal{U}$ with $|U| < \kappa$. Since $\mathcal{U}$ is a non-principal ultrafilter and $U = \{x \} \cup (U\setminus \{x\})$ for every $x \in U$, we have $U \setminus \{ x \} \in \mathcal{U}$ for every $x \in U$. Since $\mathcal{U}$ is $\kappa$-complete, we infer that $\emptyset = \bigcap_{x \in U} U \setminus \{x\} \in \mathcal{U}$, which contradicts the definition of a filter.
\end{proof}

We close our considerations with a generalization of Theorem \ref{3.3}. Indeed, the special case $\Lambda = \aleph_0$ of the following result is exactly the mentioned theorem. Here we state it for arbitrary measurable cardinals.

A cardinal $\kappa$ is called a \textit{limit cardinal}, if $\kappa = \aleph_\delta$ for a non-zero ordinal $\delta$ that is a limit, i.e., that has no predecessor.

\begin{theorem}
Let $\Lambda$ be a measureable cardinal and $\mathcal{U}$ a non-principal $\Lambda$-complete ultrafilter on $\Lambda$.
Moreover, let $(H_\lambda)_{\lambda \in \Lambda}$ be a family of commutative cancellative monoids such that $H_\lambda$ satisfies the following property for all $\lambda \in \Lambda$:

(F$_\Lambda$) For every multiset $L \subseteq \mathbb{N}_{\geq 2}$ with $\#_L(\ell) < \Lambda$ for all $\ell \in \mathbb{N}_{\geq 2}$, there exists $r_\lambda \in H_\lambda$ such that for all $\ell \in \mathbb{N}_{\geq 2}$ we have that $\#_{L(r_\lambda)}(\ell) \geq \#_L(\ell)$ and, if $\#_L(\ell)$ is finite, then $\#_{L(r_\lambda)}(\ell) = \#_L(\ell)$.

Then, for every multiset $L \subseteq \mathbb{N}_{\geq 2}$ such that $\#_L(\ell) \leq \Lambda$ is either finite or countable or a limit cardinal for all $\ell \in \mathbb{N}_{\geq 2}$, there exists $r^*$ in the ultraproduct $H^* = \prod_{\lambda \in \Lambda}^\mathcal{U} H_\lambda$ such that for all $\ell \in \mathbb{N}_{\geq 2}$
\begin{itemize}
\item[(i)] $\#_{L(r^*)}(\ell) \geq \#_L(\ell)$ and
\item[(ii)] if $\#_L(\ell)$ is finite then $\#_{L(r^*)}(\ell) = \#_L(\ell)$.
\end{itemize}
\end{theorem}

\begin{proof}
Let $L \subseteq \mathbb{N}_{\geq 2}$ be a multiset with the property that $\#_{L}(\ell) \leq \Lambda$ is either finite or countable or a limit cardinal for all $\ell \in \mathbb{N}_{\geq 2}$. Since $L$ is a countable union of multisets containing one single element and this element has multiplicity at most $\Lambda$, we have $|L| \leq \Lambda$.

\textbf{Case 1:} $|L| < \Lambda$. Then, in particular, $\#_L(\ell) < \Lambda$ for all $\ell \in \mathbb{N}_{\geq 2}$. The assertion follows immediately by the assumption (F$_\Lambda$) together with Proposition \ref{4.1}(2) and Theorem~\ref{2.4}.

\textbf{Case 2:} $|L| = \Lambda$. Let $f : \Lambda \to L$ be a bijective map, in the sense that every element $\ell \in L$ is the image of exactly $\#_L(\ell)$ elements of $\Lambda$. Define
\begin{align*}
L_\lambda = \{f(\alpha) \mid \alpha \leq \lambda \}
\end{align*}
for all $\lambda \in \Lambda$. Clearly, $\#_{L_\lambda}(\ell) \leq |L_\lambda| = |\lambda| < \Lambda$ for all $\lambda \in \Lambda$ and $\ell \in \mathbb{N}_{\geq 2}$. By assumption, we can pick $r_\lambda \in D_\lambda$ such that $L(r_\lambda) = L_\lambda$. We claim that $r^* = (r_\lambda)^*_\lambda \in H^*$ has the aimed property. To show this, let $\ell \in \mathbb{N}_{\geq 2}$.

If $\#_L(\ell)$ is finite, then, since $f$ is bijective, there exists $\lambda \in \Lambda$ such that for all $\alpha \in \Lambda$ with $\alpha \geq \lambda$ we have $\#_{L_\alpha}(\ell) = \#_L(\ell)$. Let $\lambda \in \Lambda$ be minimal with this property, that is, $\#_{L_\mu}(\ell) < \#_L(\ell)$ for all $\mu \in \lambda$. Since $\mathcal{U}$ is $\Lambda$-complete and $|\lambda| < \Lambda$, it follows by Remark~\ref{4.5} that $\lambda = \{ \mu \in \Lambda \mid \#_{L_\mu}(\ell) < \#_L(\ell) \} \notin \mathcal{U}$. Therefore $\{ \mu \in \Lambda \mid \#_{L_\mu}(\ell) = \#_L(\ell) \} \in \mathcal{U}$. We infer by Theorem \ref{2.4} that $\#_{L(r^*)}(\ell) = \#_L(\ell)$.

Now, if $\#_L(\ell)$ is infinite, by Proposition \ref{4.1}(2) it suffices to show that for all $\lambda < \#_L(\ell)$ we have $\{ \mu \in \Lambda \mid \#_{L_\mu}(\ell) > \lambda \} \in \mathcal{U}$. Assume to the contrary that there exists a cardinal $\lambda < \#_L(\ell)$ such that $U = \{ \mu \in \Lambda \mid \#_{L_\mu}(\ell) \leq \lambda \} \in \mathcal{U}$. Since $\mathcal{U}$ is $\Lambda$-complete, it follows that $|U| = \Lambda$. For all $\mu_1, \mu_2 \in \Lambda$ with $\mu_1 \leq \mu_2$ we have that $L_{\mu_1} \subseteq L_{\mu_2}$ and therefore $\#_{L_{\mu_2}}(\ell) \leq \lambda$ implies $\#_{L_{\mu_1}}(\ell) \leq \lambda$, so $\mu_2 \in U$ implies $\mu_1 \in U$. If $U$ contains a maximal element $\mu$ then $U = \mu + 1$ is an ordinal. Otherwise, $U = \bigcup_{\mu \in U} \mu$ as a union of ordinals is also an ordinal. Since $|U| = \Lambda$, it follows that $U = \Lambda$.

Set $P = \{ \mu \in \Lambda \mid f(\mu) = \ell \}$ and $P_\mu = P \cap \mu$ for all $\mu \in \Lambda$. Then $|P| = \#_L(\ell)$ and $|P_\mu| = \#_{L(r_\mu)}(\ell) \leq \lambda < \#_L(\ell)$ for all $\mu \in \Lambda = U$. Since $P_\mu \subseteq P_{\mu +1} \subseteq P_\mu \cup \{\mu\}$ for all $\mu \in \Lambda$, in view of cardinalities, we might as well assume that the $P_\mu$ are indexed over $\#_L(\ell)$, that is, we have $(P_\mu)_{\mu \in \#_L(\ell)}$ with $P_{\mu +1} = P_\mu \cup \{\mu\}$. Then, since $|P_\mu| \leq \lambda < \#_L(\ell)$, $|P| = \#_L(\ell)$ and $P = \bigcup_{\mu \in \#_L(\ell)} P_\mu$, it follows that $\#_L(\ell) = \lambda^+$ is a successor cardinal. This is a contradiction to the assumption that $\#_L(\ell)$ either is countable or a limit cardinal.
\end{proof}

\section{A new proof of a theorem related to additive combinatorics}\label{section:application}

In order to illustrate how our results can be applied in somewhat unexpected areas of mathematics and to propose a general recipe for such applications, we want to give a new proof, using ultraproducts, for a theorem of Geroldinger, W.A. Schmid and Q. Zhong from 2017. It deals with a concept present in both, additive combinatorics and the theory of non-unique factorizations, namely, with the monoid of zero-sum sequences.

Let $(G,+,0)$ be an Abelian group. We denote by $\mathcal{F}(G)$ the free Abelian monoid on the basis~$G$. The elements of $\mathcal{F}(G)$ are finite unordered sequences of elements of $G$, where repetition is allowed. We denote such sequences by $U = g_1 \bdot g_2 \bdot \cdots \bdot g_n$, where $g_i \in G$, and set $|U| = n$ the \textit{length} of $U$. The monoid operation on $\mathcal{F}(G)$ is concatenation of sequences and the neutral element is the empty sequence.

There is a homomorphism $\sigma: \mathcal{F}(G) \to G$ given by $\sigma(g_1 \bdot g_2 \bdot \cdots \bdot g_n) = g_1 + g_2 + \cdots + g_n$. The submonoid $\mathcal B(G) := \sigma^{-1}(\{0\})$ of $\mathcal{F}(G)$ is called the \textit{monoid of zero-sum sequences} over $G$. It plays a prominent role in the study of non-unique factorizations in Krull monoids and, in particular, in normal Noetherian integral domains, see~\cite[Chapters 3, 5, 6, 7]{GHK}.

\begin{theorem}\cite[Theorem 3.7]{application}\label{theorem:application}
Let $L$ be a finite non-empty set of integers $\geq 2$. Then there are only finitely many finite Abelian groups $G$ (up to isomorphism) such that $L$ is not the set of lengths of an element in $\mathcal{B}(G)$.
\end{theorem}

Before we do this, let us describe an idea for ultraproduct constructions in a setting more general than just monoids of zero-sum sequences, including, for instance, many suitably graded rings like polynomial rings and rings of integer-valued polynomials~\cite{Cahen-Chabert}.
We consider commutative cancellative monoids $H$ with a submultiplicative map $\deg: H \to \N_0 \cup \{\infty\}$ in the sense that $\deg(ab) \leq \deg(a) + \deg(b)$. We use the convention $n+ \infty = \infty + n = \infty$ for each $n \in \N_0 \cup \{\infty \}$ and call $\deg$ a \textit{degree} on $H$. If $\deg(ab) = \deg(a) + \deg(b)$ for all $a,b \in H$, we say that $\deg$ is \textit{multiplicative}. Examples of such functions are given by the classical degree on the multiplicative monoid of a polynomial ring over an integral domain and by the map $| \ . \ |: \mathcal{B}(G) \to \N_0$.

A submonoid $D \subseteq H$ is called \textit{divisor-closed} if $d = ab$ implies $a \in D$ for all $d \in D$ and $a,b \in H$. 

\begin{remark}\label{remark:divisor-closed}
It is easy to see that, if $D$ is a divisor-closed submonoid of $H$ and $d \in D$, then $d$ is an atom of $D$ if and only if $d$ is an atom of $H$. Moreover, $L_H(d) = L_D(d)$ and, in particular, every set of lengths of an element in $D$ is a set of lengths of an element in $H$.
\end{remark}

\begin{lemma}\label{lemma:deg}
Let $H$ be a commutative cancellative monoid and $\deg: H \to \N_0 \cup \{\infty\}$ a degree on $H$. Then $\deg^{-1}(\N_0)$ is a submonoid of $H$. If $\deg$ is multiplicative then this submonoid is divisor closed.
\end{lemma}

\begin{proof}
This is an elementary computation based on all relevant definitions.
\end{proof}

\begin{example}\label{example:protoproduct}
Let $(H_\lambda)_{\lambda \in \Lambda}$ be a family of commutative cancellative monoids and, for each $\lambda$, let $\deg_\lambda$ be a multiplicative degree on $H_\lambda$. Let $\mathcal{U}$ be an ultrafilter on $\Lambda$.

For $h^* \in H^* = \prod_\lambda^\mathcal{U} H_\lambda$, we define $\deg(h^*)= n$ if $\{\lambda \mid \deg(h_\lambda) = n\} \in \mathcal{U}$ and $\deg(h^*)= \infty$ else. $\deg$ defines a multiplicative degree on $H^*$ and we call $H^\flat = \deg^{-1}(\N_0)$ the \textit{protoproduct} of the $H_\lambda$ (with respect to $\deg$), cf.~\cite[Chapter 12]{Schoutens}. It is a divisor-closed submonoid of $H^*$ by Lemma~\ref{lemma:deg}.
\end{example}

\begin{lemma}\label{lemma:toUP}
Let $L \subseteq \mathbb{N}_{\geq 2}$. In the notation of Example~\ref{example:protoproduct}, if $\{\lambda \in \Lambda \mid L \notin \mathcal{L}(H_\lambda) \} \in \mathcal{U}$ then $L \notin \mathcal{L}(H^\flat)$.
\end{lemma}

\begin{proof}
We prove the logical contraposition of the statement. Let $U^* \in H^\flat$ with $L = L(U^*)$. Now for $n \in L$ there exist atoms $U_1^*, \hdots, U_n^* \in H^\flat$ such that $U^* = U_1^* \cdots U_n^*$. So there exists $V \in \mathcal{U}$ such that for all $\lambda \in V$ we have that $U_\lambda = U_1^{(\lambda)} \cdots U_n^{(\lambda)}$ and $U_1^{(\lambda)}, \hdots , U_n^{(\lambda)} \in \mathcal{A}(\mathcal B(G_\lambda))$ by Lemma~\ref{lemma:deg}, Remark~\ref{remark:divisor-closed}, Example~\ref{example:protoproduct} and Lemma~\ref{2.2}. Therefore $\{\lambda \in \Lambda \mid n \in L(U_\lambda) \} \in \mathcal{U}$. By intersection, we have that $\{\lambda \in \Lambda \mid L \subseteq L(U_\lambda) \} \in \mathcal{U}$. \\
Now either $\{\lambda \in \Lambda \mid L = L(U_\lambda) \} \in \mathcal{U}$ or $\{\lambda \in \Lambda \mid L \subsetneqq L(U_\lambda) \} \in \mathcal{U}$. Assume to the contrary that $V := \{\lambda \in \Lambda \mid L \subsetneqq L(U_\lambda) \} \in \mathcal{U}$. Then for every $\lambda \in V$ there exists $n_\lambda \in L(U_\lambda) \setminus L$. Then $\{n_\lambda \mid \lambda \in V\} = \{n_1,\ldots ,n_k\}$ is finite, since for some $N \in \mathbb{N}$ we have that $\deg(U_\lambda) \leq N$ for all $\lambda \in V$, because $U^* \in H^\flat$. Therefore we obtain a decomposition of $V$, namely $V$ is the union of $V_j := \{ \lambda \in V \mid n_j \in L(U_\lambda)\}$ for $j \in \{1,\ldots,k\}$. We conclude that there exists some $j \in \{1,\ldots,k\}$ such that $V_j \in \mathcal{U}$. So for every $\lambda \in V_j$, there exists a factorization of $U_\lambda$ into $n_j$ atoms, contradicting the fact that $n_j \notin L$.
\end{proof}

\begin{example}\label{example:B(G)}
Let $(G_\lambda)_{\lambda \in \Lambda}$ be a family of Abelian groups, $G^* = \prod_\lambda^\mathcal{U} G_\lambda$ and $H_\lambda = \mathcal{B}(G_\lambda)$. On $H_\lambda$ we use as a degree the length function $| \ . \ |$. It is multiplicative. The map
\begin{align*}
\mathcal{B}(G^*) &\to H^\flat \\
g_1^* \bdot \cdots \bdot g_n^* &\mapsto (g_1^{(\lambda)} \bdot \cdots \bdot g_n^{(\lambda)})_{\lambda \in \Lambda}
\end{align*}
is a monoid isomorphism.
\end{example}

\begin{proof}[Proof of Theorem~\ref{theorem:application}]
Assume to the contrary that there are infinitely many pairwise non-isomorphic finite Abelian groups $G$ such that $L$ is not the set of lengths of an element in $\mathcal{B}(G)$. In particular, for each positive integer $n$, we find a group $G_n$ of cardinality at least $n$ such that $L$ is not the set of lengths of an element in $\mathcal{B}(G_n)$. Let $\mathcal{U}$ be a non-trivial ultrafilter on the set of all positive integers, $H_n = \mathcal B(G_n)$, $G^* = \prod^\mathcal{U} G_n$, $H^* = \prod^\mathcal{U} H_n$ and let $H^\flat$ the protoproduct as in Examples~\ref{example:protoproduct} and~\ref{example:B(G)}.

As the cardinalities of the $G_n$ are unbounded, $G^*$ is an infinite Abelian group and, hence, by the Theorem of Kainrath~\cite{Kainrath}, $L$ is the set of lengths of an element in $\mathcal{B}(G^*) \cong H^\flat$, where the isomorphism comes from Example~\ref{example:B(G)}. The same is true for $H^*$ which contains $H^\flat$ as a divisor closed submonoid by Example~\ref{example:protoproduct}. This is a contradiction to Lemma~\ref{lemma:toUP}. 
\end{proof}

\begin{remark}\label{remark:examplesproto}
We want to note that the method developed in this section could have various applications to the study of factorizations in monoids and integral domains. For many natural objects with a degree, the protoproduct of a family of such objects is again an object of this type. We want to give three examples. Let $(D_\lambda)_\lambda$ be a family of integral domains.
\begin{enumerate}
\item The protoproduct of the family of polynomial rings $D_\lambda[X]$ (with respect to the usual degree), where $X$ is an indeterminate over all the $D_\lambda$, is isomorphic to the polynomial ring over the ultraproduct $D^*$ of the $D_\lambda$.

\item In an analogous way, one can define a degree on a semigroup ring $R[S]$. For $f \in D[S]$, let $\deg(f)$ be the cardinality of the support of $f$, that is, the number of non-zero coefficients of $f$. The protoproduct of a family of semigroup rings $R_\lambda[S_\lambda]$ with respect to this degree is just the semigroup ring of the corresponding ultraproducts of the $R_\lambda$ and the $S_\lambda$.

\item For an integral domain $D$ with quotient field $K$, we call $\text{Int}(D) = \{f \in K[X] \mid f(D) \subseteq D\}$ the \textit{ring of integer-valued polynomials} on $D$, see~\cite{Cahen-Chabert}. The protoproduct of the rings $\text{Int}(D_\lambda)$ (with respect to the usual degree) is isomorphic to the ring $\text{Int}(D^*)$.
\end{enumerate}

The ring $\text{Int}(D)$ is a well-studied object with respect to factorizations when $D$ is a Dedekind domain (or at least a Krull domain). But, also in this context, there are many open questions. In view of the applications in this section, it seems sensible to initiate the study of factorizations in rings of integer-valued polynomials on ultraproducts of Dedekind domains. Such ultraproducts were just recently studied with respect to their prime ideal structure~\cite{primes}.
\end{remark}

\section{Protoproducts of Krull monoids}

Krull monoids are prominent objects in the study of non-unique factorizations. The reason is that their system of sets of lengths is completely determined by their divisor class group and the distribution of prime divisors in the classes, see~\cite[Section 3.4]{GHK}. However, Remark~\ref{2.3}(2) shows that an ultraproduct of Krull monoids is not necessarily Krull (actually it is not in every non-trivial case). But there is no reason for despair: An ultraproduct of Krull monoids always contains a large divisor closed submonoid that is Krull. If, for instance, all the Krull monoids we want to consider are actually Krull domains, we can even determine the class group and the distribution of prime divisors.
To see this, we need a sensible degree function on a Krull monoid. 

\begin{definition}
Let $H$ be a Krull monoid and let $\phi: H \to \mathcal{F}(P)$ be a divisor theory for $H$, see~\cite[Section 2.4]{GHK}. For $a \in H$, let $\phi(a) = p_1\cdots p_n$ with $p_i \in P$. We set $\deg(a) = n$ and call $\deg: H \to \N_0$ the \textit{canonical degree} on $H$.
\end{definition}

Since $\phi$ is a homomorphism, $\deg$ is indeed a multiplicative degree on $H$. Moreover, as a divisor theory is unique up to isomorphism, $\deg$ is well-defined. It is a standard result that a monoid $\mathcal{B}(G)$ of zero-sum sequences is always Krull. Note that the canonical degree on $\mathcal{B}(G)$ is exactly the length function as introduced in the preceding section.

Let $\phi: H \to \mathcal{F}(P)$ be a divisor theory for the Krull monoid $H$. $H$ satisfies the \textit{approximation property} if, for all $p_1,\ldots,p_n \in P$ and non-negative integers $e_1,\ldots,e_n$, there exists $a \in H$ such that $p_i^{e_i}$ is the highest power of $p_i$ dividing $\phi(a)$ in $\mathcal{F}(P)$ for every $i$. See the book on ideal systems by Halter-Koch \cite[Chapter 25 \& 26.4]{idealsystems} for more on approximation properties of Krull monoids.

\begin{remark}\label{remark:approx}
\begin{enumerate}
\item Krull domains always satisfy the approximation property. This is the classical Approximation Theorem for Krull domains, see~\cite{Gilmer}.
\item  Moreover, if $H$ has the approximation property, it is easy to see that, for every $p \in P$, there exist $a,b \in H$ such that $p = \gcd(\phi(a),\phi(b))$.
\end{enumerate}
\end{remark}

\begin{theorem}\label{theorem:Krull}
A protoproduct $H^\flat$ of a family $(H_\lambda)_{\lambda\in \Lambda}$ of Krull monoids with respect to their canonical degrees is a Krull monoid.

Moreover, if all the $H_\lambda$ satisfy the approximation property then the well-defined map
\begin{align*}
\phi: H^\flat &\to \mathcal{F}(P^*) \\
        a^* &\mapsto (\phi_\lambda(a_\lambda))^*_\lambda
\end{align*}
is a divisor theory for $H^\flat$, where $\phi_\lambda: H_\lambda \to \mathcal{F}(P_\lambda)$ is a divisor theory for $H_\lambda$ and $\mathcal{F}(P^*)$, the free Abelian monoid over the ultraproduct $P^*$ of the $P_\lambda$, is isomorphically identified with the corresponding protoproduct of the $\mathcal{F}(P_\lambda)$.

In particular, the divisor class group of $H^\flat$ is the corresponding ultraproduct of the divisor class groups $\mathcal{C}_v(H_\lambda)$ of the $H_\lambda$ and the set containing prime divisors is the corresponding ultraproduct of the subsets of the $\mathcal{C}_v(H_\lambda)$ containing prime divisors.
\end{theorem}

\begin{proof}
It can be easily seen that the protoproduct of the $\mathcal{F}(P_\lambda)$ (with respect to the degree given by the length of the unique factorization into primes) is isomorphic to $\mathcal{F}(P^*)$.

The map $\phi$ is a well-defined divisor homomorphism, even without imposing the approximation property on the components. This is just by definition of $H^\flat$. Therefore, using~\cite[Theorem 2.4.8]{GHK}, $H^\flat$ is a Krull monoid.

When each $H_\lambda$ satisfies the approximation property, $\phi$ is a divisor theory by Remark~\ref{remark:approx}(2). Indeed, given $p^* = (p_\lambda)^* \in P^*$, we choose $a_\lambda,b_\lambda \in H_\lambda$ such that $p_\lambda = \gcd(\phi_\lambda(a_\lambda),\phi_\lambda(b_\lambda))$ for each $ \lambda$. Then $p^* = \gcd(\phi(a^*),\phi(b^*))$.

All other statement follow immediately from $\phi$ being a divisor theory.
\end{proof}

\begin{remark}
We want to note that the approximation property is not strictly necessary for the second half of Theorem~\ref{theorem:Krull}. As is evident from the proof, it suffices that there exists a uniform positive integer $n$ such that for every $\lambda$ and every $p \in P_\lambda$ there exist $a_1,\ldots,a_n \in H_\lambda$ such that $p = \gcd(\phi(a_1),\ldots,\phi(a_n))$. In the case of the approximation property, this $n$ can be chosen equal to $2$.
\end{remark}

\begin{corollary}\label{corollary:rings}
For each $\lambda \in \Lambda$, let $R_\lambda$ be a Krull domain with divisor class group $G_\lambda$ and set $G_{\lambda 0}$ containing prime divisors. Let $\mathcal{U}$ be an ultrafilter on $\Lambda$.

Then the protoproduct $R^\flat$ with respect to $\mathcal{U}$ and to the canonical degrees on the $R_\lambda$ is a Krull monoid with divisor class group $\prod^\mathcal{U} G_\lambda$ and set $\prod^\mathcal{U} G_{\lambda 0}$ containing prime divisors.
\end{corollary}

\begin{proof}
The statement immediately follows from Theorem~\ref{theorem:Krull} and Remark~\ref{remark:approx}(1).
\end{proof}

\begin{remark}\label{remark:notring}
Note that, in order for $R^\flat$ from Corollary~\ref{corollary:rings} to be a ring, it is necessary and sufficient that there exists a uniform bound for the number of prime divisors of $a_\lambda -b_\lambda$ just depending on the numbers of prime divisors of $a_\lambda$ and $b_\lambda$ that is also independent of $\lambda \in \Lambda$ and the choice of $a_\lambda, b_\lambda \in R_\lambda$.

We want to illustrate that this is not even the case in the supposedly easiest situation where $R_\lambda = \Z$ for every $\lambda \in \Lambda$. However, a main ingredient of this argument are the finite residue fields of $\Z$. It is not clear what happens for component rings $R_\lambda$ with infinite residue fields.

Our argument comes from~\cite{Number-primediv}. We repeat it hear for the sake of completeness and show that the difference of two prime powers can have arbitrarily many prime divisors. Let $p,q,p_1,\ldots,p_k$ distinct prime numbers and set $N_k = p_1 \cdots p_k$. Choose $x,y$ such that $p^x \equiv q^y \equiv 1$ mod $N_k$. Then $N_k$ divides $p^x - q^y$ and hence this number has at least $k$ prime divisors.
\end{remark}

Despite Remark~\ref{remark:notring}, it might be interesting for future research to study protoproducts with respect to factorization. As we saw in the natural examples of Remark~\ref{remark:examplesproto}, the construction possibly yields rings again and good knowledge of its arithmetic can be extremely helpful to study classical objects.

\section*{Acknowledgements} 
The proof of Theorem~\ref{theorem:application} is the product of lively discussions with my friend and colleague Victor Fadinger-Held. Thankful for the years of study and research we shared, I would like to dedicate this paper to him.

\bibliographystyle{amsplainurl}
\bibliography{bibliography}

\providecommand{\bysame}{\leavevmode\hbox to3em{\hrulefill}\thinspace}
\providecommand{\MR}{\relax\ifhmode\unskip\space\fi MR }
\providecommand{\MRhref}[2]{%
  \href{http://www.ams.org/mathscinet-getitem?mr=#1}{#2}
}
\providecommand{\href}[2]{#2}
\begin{thebibliography}{10}

\bibitem{Ax-Kochen}
James Ax and Simon Kochen, \emph{Diophantine problems over local fields {I},
  {II}}, American Journal of Mathematics \textbf{87} (1965), no.~3, 605--630,
  631--648.

\bibitem{algtop}
Tobias Barthel, Tomer~M. Schlank, and Nathaniel Stapleton, \emph{Chromatic
  homotopy theory is asymptotically algebraic}, Invent. Math. \textbf{220}
  (2020), no.~3, 737--845.

\bibitem{algtop1}
\bysame, \emph{Monochromatic homotopy theory is asymptotically algebraic}, Adv.
  Math. \textbf{393} (2021), 44, Id/No 107999.

\bibitem{Smertnig-BF}
Jason~P. Bell, Ken Brown, Zahra Nazemian, and Daniel Smertnig,
  \href{http://dx.doi.org/10.1016/j.jalgebra.2023.01.023}{\emph{On
  noncommutative bounded factorization domains and prime rings}}, J. Algebra
  \textbf{622} (2023), 404--449.

\bibitem{Cahen-Chabert}
Paul-Jean Cahen and Jean-Luc Chabert,
  \href{https://books.google.at/books?id=OdLxBwAAQBAJ}{\emph{Integer-{V}alued
  {P}olynomials}}, American Mathematical Society Translations, American
  Mathematical Society, 1997.

\bibitem{Carlitz-halffactorial}
Leonard Carlitz,
  \href{https://www.ams.org/journals/proc/1960-011-03/S0002-9939-1960-0111741-2/}{\emph{A
  characterization of algebraic number fields with class number two}}, Proc.
  Amer. Math. Soc. \textbf{11} (1960), 391--392.

\bibitem{mod}
C.C. Chang and H.J. Keisler, \emph{Model {T}heory}, North-Holland Publishing
  Company, 1990.

\bibitem{weaklyKrull2}
Gyu~W. Chang, Victor Fadinger, and Daniel Windisch,
  \href{http://dx.doi.org/10.2140/pjm.2022.318.433}{\emph{Semigroup rings as
  weakly {K}rull domains}}, Pacific Journal of Mathematics \textbf{318} (2022),
  433--452.

\bibitem{Laura-Salvo}
Laura Cossu and Salvatore Tringali,
  \href{http://dx.doi.org/10.1016/j.jalgebra.2023.04.014}{\emph{Factorization
  under local finiteness conditions}}, J. Algebra \textbf{630} (2023),
  128--161.

\bibitem{Ebbinghaus}
H.D. Ebbinghaus, \emph{Einf{\"u}hrung in die {M}engenlehre}, Spektrum
  Akademischer Verlag, 2003.

\bibitem{FFW}
Victor Fadinger, Sophie Frisch, and Daniel Windisch, \emph{Integer-valued
  polynomials on discrete valuation rings of global fields with prescribed
  lengths of factorizations}, Monatsh. Math. \textbf{202} (2023), 773--789.

\bibitem{affine}
Victor Fadinger and Daniel Windisch, \emph{On the distribution of prime
  divisors in class groups of affine monoid algebras}, Communications in
  Algebra \textbf{50} (2022), no.~7, 2973--2982.

\bibitem{KrullPrimeDiv}
\bysame, \emph{On the distribution of prime divisors in {K}rull monoid
  algebras}, Journal of Pure and Applied Algebra \textbf{226:4} (2022).

\bibitem{Int-Krull-Prime}
\bysame, \emph{Lengths of factorizations of integer-valued polynomials on
  {K}rull domains with prime elements}, 2023.

\bibitem{primes}
Carmelo~A. Finocchiaro, Sophie Frisch, and Daniel Windisch, \emph{Prime ideals
  in infinite products of commutative rings}, accepted in Communications in
  Contemporary Mathematics.

\bibitem{Frisch}
Sophie Frisch, Sarah Nakato, and Roswitha Rissner, \emph{Sets of lengths of
  factorizations of integer-valued polynomials on {D}edekind domains with
  finite residue fields}, J. Algebra \textbf{528} (2019), 231--249.

\bibitem{Geroldinger-Goebel}
Alfred Geroldinger and R{\"u}diger G{\"o}bel, \emph{Half-factorial subsets in
  infinite abelian groups}, Houston J. Math. \textbf{29} (2003), no.~4,
  841--858.

\bibitem{GHK}
Alfred Geroldinger and Franz Halter-Koch, \emph{Non-{U}nique {F}actorizations:
  Algebraic, {C}ombinatorial and {A}nalytic {T}heory}, Chapman \& Hall/CRC Pure
  and Applied Mathematics, CRC Press, 2006.

\bibitem{Geroldinger-Crelle}
Alfred Geroldinger, Franz Halter-Koch, and Jerzy Kaczorowski, \emph{Non-unique
  factorizations in orders of global fields.}, Journal für die reine und
  angewandte Mathematik \textbf{459} (1995), 89--118.

\bibitem{application}
Alfred Geroldinger, Wolfgang~A. Schmid, and Qinghai Zhong, \emph{Systems of
  sets of lengths: {T}ransfer {K}rull monoids versus weakly {K}rull monoids},
  Rings, Polynomials, and Modules (Marco Fontana, Sophie Frisch, Sarah Glaz,
  Francesca Tartarone, and Paolo Zanardo, eds.), Springer International
  Publishing, Cham, 2017, pp.~191--235.

\bibitem{Gilmer_Semigroup-Rings}
Robert Gilmer,
  \href{https://books.google.at/books?id=sOPNfkp-Le8C}{\emph{Commutative
  {S}emigroup {R}ings}}, Chicago Lectures in Mathematics, University of Chicago
  Press, 1984.

\bibitem{Gilmer}
\bysame,
  \href{https://books.google.at/books?id=PGEOzQEACAAJ}{\emph{Multiplicative
  {I}deal {T}heory}}, Queen's papers in pure and applied mathematics, Queen's
  University, 1992.

\bibitem{idealsystems}
Franz Halter-Koch,
  \href{https://books.google.at/books?id=2ZApizCs2WUC}{\emph{Ideal {S}ystems:
  {A}n {I}ntroduction to {M}ultiplicative {I}deal {T}heory}}, Chapman \&
  Hall/CRC Pure and Applied Mathematics, Taylor \& Francis, 1998.

\bibitem{Number-primediv}
Zander (https://math.stackexchange.com/users/25711/zander), \emph{Number of
  prime factors of difference of two numbers}, Mathematics Stack Exchange,
  URL:https://math.stackexchange.com/q/420809 (version: 2013-06-15).

\bibitem{Franzi}
Franziska Jahnke and Konstantinos Kartas, \emph{Beyond the
  fontaine-wintenberger theorem}, 2023.

\bibitem{Jech}
Thomas Jech, \emph{Set theory: The third millennium edition, revised and
  expanded}, Springer Monographs in Mathematics, Springer Berlin Heidelberg,
  2007.

\bibitem{KainrathDistribution}
Florian Kainrath, \emph{The distribution of prime divisors in finitely
  generated domains}, manuscripta mathematica \textbf{100} (1999), 203--212.

\bibitem{Kainrath}
Florian Kainrath, \href{http://eudml.org/doc/210702}{\emph{Factorization in
  {K}rull monoids with infinite class group}}, Colloquium Mathematicae
  \textbf{80} (1999), no.~1, 23--30.

\bibitem{Balint}
Balint Rago, \emph{A characterization of half-factorial orders in algebraic
  number fields}, 2023.

\bibitem{Schoutens}
H.~Schoutens, \href{https://books.google.at/books?id=ohbdke-cBxwC}{\emph{The
  {U}se of {U}ltraproducts in {C}ommutative {A}lgebra}}, Lecture Notes in
  Mathematics, no. 1999, Springer, 2010.

\bibitem{Smertnig-survey}
Daniel Smertnig,
  \href{http://dx.doi.org/10.1007/978-3-319-38855-7_15}{\emph{Factorizations of
  elements in noncommutative rings: a survey}}, Multiplicative ideal theory and
  factorization theory. Commutative and non-commutative perspectives. Selected
  papers based on the presentations at the meeting `Arithmetic and ideal theory
  of rings and semigroups', Graz, Austria, September 22--26, 2014, Cham:
  Springer, 2016, pp.~353--402.

\bibitem{Salvo1}
Salvatore Tringali,
  \href{http://dx.doi.org/10.1016/j.jalgebra.2022.03.023}{\emph{An abstract
  factorization theorem and some applications}}, J. Algebra \textbf{602}
  (2022), 352--380.

\bibitem{Salvo2}
\bysame, \href{http://dx.doi.org/10.1017/S0305004123000269}{\emph{A
  characterisation of atomicity}}, Math. Proc. Camb. Philos. Soc. \textbf{175}
  (2023), no.~2, 459--465.

\bibitem{classgroup-complete-int}
Daniel Windisch, \emph{On local divisor class groups of complete
  intersections}, 2023.

\end{thebibliography}
 
\vspace{1cm} 
 
\noindent
\textsc{Daniel Windisch, Department of Analysis and Number Theory (5010), Technische Universität Graz, Kopernikusgasse 24, 8010 Graz, Austria} \\
\textit{E-mail address}: \texttt{dwindisch@math.tugraz.at}

\end{document}